\documentclass[11pt]{amsart}
\usepackage{epsfig, amsfonts,amssymb, amsmath}

\usepackage{indentfirst}
\usepackage{amssymb}
\usepackage{amsmath}
\usepackage{epsfig}
\usepackage{xcolor}
\usepackage{picinpar}
\usepackage{float}
\usepackage{MnSymbol,wasysym}
\usepackage{mathtools}

\usepackage[utf8]{inputenc}
\usepackage{latexsym}
\usepackage{graphicx}

\usepackage[normalem]{ulem}
\usepackage{soul}

\catcode`@=11 \@addtoreset{equation}{section}
\renewcommand\theequation{\thesection.\@arabic\c@equation}
\catcode`@=12

\newcommand{\RR}{\mathbb{R}}
\newcommand{\Z}{\mathbb{Z}}
\newcommand{\N}{\mathbb{N}}
\newcommand{\R}{\mathbb{R}}

\newcommand{\dd}{{\rm\,d}}

\expandafter\chardef\csname pre amssym.def
at\endcsname=\the\catcode`\@ \catcode`\@=11
\def\undefine#1{\let#1\undefined}
\def\newsymbol#1#2#3#4#5{\let\next@\relax
	\ifnum#2=\@ne\let\next@\msafam@\else
	\ifnum#2=\tw@\let\next@\msbfam@\fi\fi
	\mathchardef#1="#3\next@#4#5}
\def\mathhexbox@#1#2#3{\relax
	\ifmmode\mathpalette{}{\m@th\mathchar"#1#2#3}%
	\else\leavevmode\hbox{$\m@th\mathchar"#1#2#3$}\fi}
\def\hexnumber@#1{\ifcase#1 0\or 1\or 2\or 3\or 4\or 5\or 6\or 7\or 8\or
	9\or A\or B\or C\or D\or E\or F\fi}

\font\teneufm=eufm10 \font\seveneufm=eufm7 \font\fiveeufm=eufm5
\newfam\eufmfam
\textfont\eufmfam=\teneufm \scriptfont\eufmfam=\seveneufm
\scriptscriptfont\eufmfam=\fiveeufm

\catcode`\@=\csname pre amssym.def at\endcsname

\newcommand{\eqn}{\begin{eqnarray}}
	\newcommand{\een}{\end{eqnarray}}

\newtheorem {Theorem}  {Theorem}

\numberwithin{Theorem}{section}
\newtheorem{Lemma}[Theorem]{Lemma}
\newtheorem{Proposition}[Theorem]{Proposition}
\newtheorem{Corollary}[Theorem]{Corollary}
\theoremstyle{definition}

\theoremstyle{remark}
\newtheorem{Remark}[Theorem]{Remark}

\newcommand{\ZZ}{{\mathbb Z}}

\begin{document}
	
\title[On Leray solutions with algebraic decay]
{On the topological size of the class of Leray solutions with algebraic decay}
	
\author[L. Brandolese]{Lorenzo Brandolese}
\address[L. Brandolese]{Institut Camille Jordan. 
Université de Lyon, Université Claude Bernard Lyon~1, 
69622 Villeurbanne Cedex. France}
\email{brandolese@math.univ-lyon1.fr}
	
\author[C. F. Perusato]{Cilon F. Perusato}
\address[C. F. Perusato]{Departamento de Matem\'atica. 
Universidade Federal de Pernambuco, Recife, PE 50740-560. Brazil}
\email{cilon.perusato@ufpe.br}
	
\author[P.R. Zingano]{Paulo R. Zingano}
\address[P.R. Zingano]{Departamento de Matemática Pura e Aplicada. 
Universidade Federal do Rio Grande do Sul, 
Porto Alegre, RS 91509-900. Brazil.}
\email{paulo.zingano@ufrgs.br}
	
		
\thanks{C. F. Perusato was partially supported by 
CAPES\:\!-PRINT-\:\!88881.311964/2018-01.} 
	
	
	
\date{\today}

%
%
%
%
	
\begin{abstract}
In 1987, Michael Wiegner in his seminal paper \cite{Wiegner1987} 
provided an important result regarding the
energy decay of Leray solutions $ \boldsymbol u(\cdot,t) $
to the
incompressible Navier-Stokes in $ \mathbb{R}^{n}$:
if the associated Stokes flows 
had their $\hspace{-0.020cm}L^{2}\hspace{-0.050cm}$ norms 
bounded by
$ O(1 + t)^{-\;\!\alpha} $
for some $ 0 < \alpha \leq (n+2)/4 $,
then the same would be true
of
$ \|\hspace{+0.020cm} \boldsymbol u(\cdot,t) 
\hspace{+0.020cm} \|_{L^{2}(\mathbb{R}^{n})} $.
The converse also holds, \\
as shown by Z.\,Skal\'ak\,\cite{Skalak2014}
and by our analysis below,
which uses a more straightforward argument.
As an application of these results, 
we discuss the genericity problem of algebraic decay
estimates for Leray solutions of the unforced Navier--Stokes equations.
In particular, 
we prove that Leray solutions 
with algebraic decay generically satisfy two-sided bounds of the form 
$(1+t)^{-\alpha}\lesssim 
\| \boldsymbol u(\cdot,t)\|_{L^2(\R^n)}
\lesssim (1+t)^{-\alpha}$.
\end{abstract}
	
\maketitle
	
\definecolor{OliveGreen1}{rgb}{0,0.6,0}
\definecolor{OliveGreen2}{cmyk}{0.64,0,0.95,0.40}

%
%
%
%

\section{Introduction}
%
%
%

An important property of Leray solutions
$ \mbox{\boldmath $u$}(\cdot,t) $
to the incompressible Navier--Stokes equations
is given by Wiegner's\hspace{-0.020cm}
\hspace{-0.020cm}
theorem\hspace{-0.020cm}
(\cite{Wiegner1987},\;\!\;\!p.\,305), 
which says, among other estimates, that
$ \|\hspace{+0.020cm} \mbox{\boldmath $u$}(\cdot,t) \hspace{+0.020cm}
\|_{L^{2}(\mathbb{R}^{n})} \hspace{-0.020cm}=\hspace{+0.020cm}
O\hspace{+0.010cm}(1 + t)^{-\;\!\alpha} $
when 
$ \|\hspace{+0.020cm} \mbox{\boldmath $v$}(\cdot,t) \hspace{+0.020cm}
\|_{L^{2}(\mathbb{R}^{n})} \hspace{-0.020cm}=\hspace{+0.020cm}
O\hspace{+0.010cm}(1 + t)^{-\;\!\alpha} $
for some $ 0 < \alpha \leq (n+2)/4 $,
where $ \mbox{\boldmath $v$}(\cdot,t) $
is the associated Stokes flow
defined in (\ref{eqn_stokes_flow}) below.
This property is also valid if
external forces are present,
under some appropriate assumptions.
%
%
A natural 
question is whether the converse property 
holds, that is,
if the knowledge 
that
$ \|\hspace{+0.020cm} \mbox{\boldmath $u$}(\cdot,t) \hspace{+0.020cm}
\|_{L^{2}(\mathbb{R}^{n})} \hspace{-0.020cm}=\hspace{+0.020cm}
O\hspace{+0.010cm}(1 + t)^{-\;\!\alpha} $
for some $ 0 < \alpha \leq (n+2)/4 $
will entail
that
$ \|\hspace{+0.020cm} \mbox{\boldmath $v$}(\cdot,t) \hspace{+0.020cm}
\|_{L^{2}(\mathbb{R}^{n})} \hspace{-0.020cm}=\hspace{+0.020cm}
O\hspace{+0.010cm}(1 + t)^{-\;\!\alpha} $
as well.
A positive answer was given by Skal\'ak's ``inverse Wiegner's theorem'' \;\cite{Skalak2014},
using an elaborated argument.

The purpose of this paper is to show a more straightforward approach
to obtain inverse Wiegner's theorem,
in any dimension $ n \geq 2 $,
with or without external forces, and to present some applications.
Our applications concern the genericity problem of
upper and lower algebraic bounds for solutions 
of the unforced Navier--Stokes equations. 

We will establish two facts: 
first of all, we prove that 
\emph{the class of solutions with an algebraic decay of the $L^2$-norm
is Baire-negligible}, in the class of Leray's solutions of the unforced Navier--Stokes equations. To do this, 
we  will endow the class of Leray's solutions of the 
unforced Navier--Stokes equations
with a topology that makes it a Baire space 
(the countable union of closed sets with empty interior has an empty interior) 
and prove that the subset of solutions with algebraic decay is meager.
This will be formalized by Theorem~\ref{th:genler} below.

The second fact that we will establish is the following: 
solutions of the unforced Navier-Stokes equations with $L^2$-algebraic
decay $O(1+t)^{-\alpha}$, with $0<\alpha<(n+2)/4$, 
\emph{generically do satisfy the two-sided bounds} 
$(1+t)^{-\alpha}\lesssim 
\|\hspace{+0.020cm} \boldsymbol u(\cdot,t) \hspace{+0.020cm}
\|_{L^2(\RR^n)}\lesssim (1+t)^{-\alpha}$.
Proving this second assertion requires endowing a Baire topology to the class of solutions with algebraic decay and establishing that the set of solutions
satisfying the corresponding lower bound is residual in the former class. 
This second stamement will be formalized by Theorem~\ref{th:genler2} below.

The result of Theorem~\ref{th:genler} is in the same spirit
of Schonbek's non-uniform decay result for solutions arising from general $L^2_\sigma$ data, described in~\cite[Section~2.3]{BrandoleseSchonbek2018}. 
Theorem~\ref{th:genler} should be compared also with the recent result~\cite[Proposition~5.4]{MR4541339},
where a similar problem was addressed for distributional 
(also called very weak) solutions of the Navier--Stokes equations. 
The method of \cite{MR4541339} relies on a subtle, 
abstract nonlinear open mapping principle: 
it is completely different than ours, that is essentially based 
on both direct and inverse Wiegner's theorem and 
on elementary topological considerations.

The lower bound problem addressed in Theorem~\ref{th:genler2} 
is closely related to the theory of decay characters, 
introduced in~\cite{MR2493562} and further
developed in \cite{MR3493117, MR3355116}.
Theorem~\ref{th:genler2} sheds new light on such a theory. 
Indeed, this theorem can be applied to quantify how large is, 
topologically,  the set of initial data for which 
a decay character does exist.

%
%
%
%
\bigskip
\paragraph{\bf Notations.}
For a function 
$\!\;\!\;\! \mbox{\bf w} = (\mbox{w}_1, \!\;\!..., \mbox{w}_n) $,
writing $\!\;\!\;\! \mbox{\bf w} \in L^{2}(\mathbb{R}^{n}) $
means that
$\!\;\!\;\! \mbox{w}_{i} \!\;\!\in\!\;\! L^{2}(\mathbb{R}^{n}) $
for every $ 1 \leq i \leq n $,
while $\!\;\!\;\!\mbox{\bf w} \in L^{2}_{\sigma}(\mathbb{R}^{n}) $
denotes that $ \!\;\!\;\! \mbox{\bf w} \in L^{2}(\mathbb{R}^{n}) $
and $  \mbox{div}\;\!\;\!\mbox{\bf w} = 0 $ 
in distributional sense
(i.e., $ \mbox{div}\;\!\;\!\mbox{\bf w} =\;\!0 $ 
in $ {\mathcal D}^{\;\!\prime}(\mathbb{R}^{n}) $).
\hspace{-0.050cm}$L^{p} \!\;\!$ norms of 
$\hspace{+0.020cm} \mbox{\bf w} \hspace{+0.010cm}$ 
are similarly defined; 
for example,
$ \|\hspace{+0.030cm}\mbox{\bf w}\hspace{+0.030cm}
\|_{L^{2}}^{\:\!2}  =
\|\hspace{+0.010cm}\mbox{w}_{1} \hspace{-0.020cm}
\|_{L^{2}}^{\:\!2} \hspace{-0.020cm}+ ... +
\|\hspace{+0.030cm}\mbox{w}_{\hspace{-0.010cm}n}\hspace{+0.001cm}
\|_{L^{2}}^{\:\!2} \hspace{-0.020cm}$,
$ \|\hspace{+0.020cm}D\hspace{+0.010cm}\mbox{\bf w}\hspace{+0.040cm}
\|_{L^{2}}^{\:\!2} \hspace{-0.040cm} =
\|\hspace{+0.020cm}\nabla\mbox{w}_{\hspace{-0.010cm}1}\hspace{+0.001cm}
\|_{L^{2}}^{\:\!2} \hspace{-0.020cm}+ ... +
\|\hspace{+0.020cm}\nabla\mbox{w}_{\hspace{-0.010cm}n}\hspace{+0.001cm}
\|_{L^{2}}^{\:\!2} \hspace{-0.020cm}$,
$ \|\hspace{+0.020cm}D^{2}\hspace{+0.010cm}\mbox{\bf w}\hspace{+0.040cm}
\|_{L^{2}}^{\:\!2} \hspace{-0.040cm} =
\sum_{\hspace{+0.015cm}i,\hspace{+0.010cm}j,\hspace{+0.020cm}\ell}
\hspace{-0.010cm}
\|\hspace{+0.020cm}D_{\hspace{-0.040cm}j} 
\hspace{+0.010cm} D_{\hspace{-0.030cm}\ell} \hspace{+0.045cm}
\mbox{w}_{\hspace{-0.020cm}i}\hspace{+0.020cm}
\|_{L^{2}}^{\:\!2} \hspace{-0.020cm} $,
and so on.
By $ \hspace{-0.020cm}\dot{H}^{m}(\mathbb{R}^{n})$ 
we indicate the homogeneous Sobolev 
space of order~$m$
(see e.g.\:\cite{Chemin2011}, p.\,25),
with
$ \|\hspace{+0.050cm}\mbox{\bf w}\hspace{+0.030cm}
\|_{\dot{H}^{m}(\mathbb{R}^{n})} 
\!= \|\hspace{+0.020cm} D^{m}\hspace{0.020cm} \mbox{\bf w}
\hspace{+0.030cm} \|_{L^{2}(\mathbb{R}^{n})} 
$.
$ \langle\;\!a, b \;\!\rangle $
or $ a \cdot b $
denote the standard inner product
of vectors $ a, b \in \mathbb{R}^{n} \hspace{-0.030cm}$.
For brevity we will often simply write 
$ \| \hspace{+0.030cm} \cdot \hspace{+0.030cm} \|$ instead of $\hspace{-0.010cm} 
\| \hspace{+0.030cm} \cdot \hspace{+0.030cm} 
\|_{L^{2}(\mathbb{R}^{n})} \hspace{-0.020cm}$.
We use $o(1)$ to indicate
a function that vanishes at infinity,
$ O(1) $ for a bounded function, and,
similarly:
$ O(g(t)) = O(1)\hspace{+0.030cm}g(t) $, 
$ O(g(t))^{\gamma} \hspace{-0.060cm}= 
O(1)\hspace{+0.030cm}g(t)^{\gamma} \!\;\!$,
$ o(g(t)^{\gamma} \hspace{-0.060cm}=
o(1) \hspace{+0.030cm}g(t)^{\gamma} \!\;\!$,
etc.
The notation $A(t)\lesssim B(t)$ 
means that there is a constant~$c>0$, independent on~$t$, 
such that $A(t)\le cB(t)$.

We denote by $\mathcal{S}'(\R^n)$ 
the space of tempered distributions and by 
$e^{t\Delta}$ the heat semigroup.
The integrals over the whole space will be denoted simply by $\int \!\:\!$, 
instead of $\int_{\R^n}$,
unless the explicit indication
of $ \mathbb{R}^{n} $ seems convenient. \\
%

%
%

%
%

%
%
%
%

\section{Inverse Wiegner's theorem for the Navier-Stokes equations} 
			
Given $ n \geq 2 $,
consider the\,(forced)\,Navier-Stokes equations 
in the space $ \mathbb{R}^{n} \!\:\!$,\\
\mbox{} \vspace{-0.600cm} \\
\begin{subequations}\label{NS}
\begin{equation}\label{NS1a}
\mbox{\boldmath $u$}_t \;\!+\,
\mbox{\boldmath $u$} \!\;\!\cdot\!\;\! 
\nabla \:\! \mbox{\boldmath $u$}
\,+\;\! \nabla p
\;=\;
\Delta\:\!\mbox{\boldmath $u$} 
\:+\, \mbox{\boldmath $f$}(\cdot,t), 
\end{equation}
\mbox{} \vspace{-0.900cm} \\
\begin{equation}\label{NS1b}
\nabla \!\;\!\cdot
\mbox{\boldmath $u$}\hspace{+0.010cm}(\cdot,t) \,=\: 0,
\end{equation}
\mbox{} \vspace{-0.950cm} \\
\begin{equation}\label{NS1c}
\mbox{\boldmath $u$}(\cdot,0) \,=\;
\mbox{\boldmath $u$}_0 \in L^{2}_{\sigma}(\mathbb{R}^{n}),
\end{equation}
\end{subequations}
\mbox{} \vspace{-0.150cm} \\
with $ \hspace{-0.030cm}\mbox{\boldmath $f$}(\cdot,t) 
\hspace{-0.030cm}\in\hspace{-0.030cm}
L^{1}((\hspace{+0.005cm} 0, \infty), 
L^{2}_{\sigma}(\mathbb{R}^{n})\hspace{+0.010cm}) $
satisfying,
for some constants
$ \hspace{+0.010cm} C_{\scriptstyle \hspace{-0.030cm}f}, 
K_{\hspace{-0.040cm} f}, 
\hspace{+0.020cm}\alpha > 0 $, \\
\mbox{} \vspace{-0.675cm} \\
\begin{equation}\label{eqn_2_f}
\|\;\! \mbox{\boldmath $f$}(\cdot,t) \;\!
\|_{\mbox{}_{\scriptstyle L^{2}(\mathbb{R}^{n})}}
\hspace{-0.030cm} \leq \hspace{+0.020cm} 
C_{\scriptstyle \hspace{-0.030cm}f} \;\!(1 + t)^{-\;\!\alpha \;\!-\;\! 1}
\quad \mbox{and } \quad
\|\;\! \mbox{\boldmath $f$}(\cdot,t) \;\!
\|_{\mbox{}_{\scriptstyle L^{n}(\mathbb{R}^{n})}}
\hspace{-0.030cm} \leq\hspace{+0.020cm}
K_{\hspace{-0.040cm} f} \hspace{+0.080cm}
t^{\;\!-\;\!\alpha \;\!-\;\! \frac{n\;\!+\;\!2}{4}}
\end{equation}
\mbox{} \vspace{-0.225cm} \\
for every $ t \!\;\!> 0 $.
\hspace{-0.040cm}(\hspace{+0.020cm}In the case $ n = 2 $, 
the second estimate in \eqref{eqn_2_f} becomes redundant.)
Let $ \hspace{+0.010cm} \mbox{\boldmath $u$}
\hspace{-0.010cm}= (\hspace{+0.010cm}u_{1}, u_{2},..., u_{n}) $
be a Leray solution
to the system \eqref{NS},
that is, 
a mapping
$ \hspace{+0.020cm} 
\mbox{\boldmath $u$}(\cdot,t) \in
C_{w}(\hspace{+0.010cm} 
[\hspace{+0.040cm} 0, \infty), L^{2}_{\sigma}(\mathbb{R}^{n})) 
\cap L^{2}((0,\infty), \dot{H}^{1}(\mathbb{R}^{n})) $
with 
$ \mbox{\boldmath $u$}(\cdot,0) = \mbox{\boldmath $u$}_{0} \hspace{-0.030cm}$
that satisfies
\eqref{NS} in weak sense
and, in addition,
the energy estimate {in its strong form}\\
\mbox{} \vspace{-0.650cm} \\
\begin{equation}\label{eqn_energy}
\|\;\!\mbox{\boldmath $u$}(\cdot,t) \;\!
\|_{\mbox{}_{\scriptstyle \!\;\!
	L^{2}(\mathbb{R}^{n})}}^{\:\!2}
\hspace{-0.090cm}+\hspace{+0.070cm}
2 \!\int_{\hspace{-0.005cm}s}^{\hspace{-0.010cm}t} 
\hspace{-0.070cm}
\|\hspace{+0.020cm} D \hspace{+0.010cm}
\mbox{\boldmath $u$}(\cdot,\tau) \;\!
\|_{\mbox{}_{\scriptstyle \!\;\!
	L^{2}(\mathbb{R}^{n})}}^{\:\!2}
\hspace{-0.040cm} \dd\tau
\hspace{+0.080cm}\leq\hspace{+0.080cm}
\|\;\!\mbox{\boldmath $u$}(\cdot,s) \;\!
\|_{\mbox{}_{\scriptstyle \!\;\!
	L^{2}(\mathbb{R}^{n})}}^{\:\!2}
\hspace{-0.090cm}+\hspace{+0.070cm}
2 \!\int_{\hspace{-0.005cm}s}^{\hspace{-0.010cm}t} 
\hspace{-0.040cm}
\hspace{-0.100cm} J(\tau)
\dd\tau
\end{equation}
\mbox{} \vspace{-0.100cm} \\
for $ s = 0 $ and almost all $ s > 0 $,
and every $ t > s $,
where
$ \hspace{-0.030cm} J(\tau) 
= \int_{\mathbb{R}^{n}} \!
\mbox{\boldmath $u$}(x,\tau)
\cdot \mbox{\boldmath $f$}(x,\tau) 
\hspace{+0.010cm} \dd x $.
Such solutions can be constructed at least for $2\le n\le 4$,
see e.g. \cite{Leray1934, RobinsonSadowski2016, Temam1984, Wiegner1987}.
When $n\ge5$, instead of assuming that $\boldsymbol u$ satisfies the energy estimate in its strong form~\eqref{eqn_energy}, following \cite{Wiegner1987}, we ask that $\boldsymbol u$
is a weak solution obtained as the limit in $L^2((0,\infty),\dot H^1(\R^n))$ and
in $L^2_{\rm{loc}}(\R^n\times \R^+)$, of a sequence of approximate solutions~$\boldsymbol u_k$ satisfying~\eqref{eqn_energy}.

Considering the solution
$ \hspace{+0.020cm} \mbox{\boldmath $v$}(\cdot,t) 
\in C(\hspace{+0.010cm} [\hspace{+0.010cm}0, \infty),
L^{2}_{\sigma}(\mathbb{R}^{n})) $
of the associated Stokes problem \\
\mbox{} \vspace{-0.750cm} \\
\begin{equation}\label{eqn_stokes_flow}
\mbox{\boldmath $v$}_t \,=\:
\Delta\:\!\mbox{\boldmath $v$}
\,+\, \mbox{\boldmath $f$}(\cdot,t),
\quad \;\;\;\,
\mbox{\boldmath $v$}(\cdot,0) 
\,=\;
\mbox{\boldmath $u$}_0,
\end{equation}
\mbox{} \vspace{-0.275cm} \\
where 
$ \hspace{+0.020cm} \mbox{\boldmath $u$}_0, \mbox{\boldmath $f$} $
are given in \eqref{NS}, 
our goal in this section 
is to derive the following result
(originally obtained in \cite{Skalak2014}
with $ n = 3 $, 
$ \hspace{-0.050cm}\boldsymbol f 
\hspace{-0.050cm}=\hspace{-0.020cm} {\bf 0} \hspace{+0.020cm}$ 
using a very different argument). 
	
%
%
%
%
%
	
\begin{Theorem}[{\sc \small Inverse Wiegner}]\label{Inverse_Wiegner}
If $\hspace{+0.070cm}
\|\;\!\mbox{\boldmath $u$}(\cdot,t)\;\!
\|_{L^{2}(\mathbb{R}^{n})}
\hspace{-0.050cm} =\hspace{+0.020cm}
O(1 +\hspace{+0.020cm}t)^{-\;\!\alpha} \hspace{-0.020cm} $
for some $\;\! 0 < \alpha \leq $ $(n+2)/4 $,
then 
$\hspace{+0.030cm}
\|\;\!\mbox{\boldmath $v$}(\cdot,t)\;\!
\|_{L^{2}(\mathbb{R}^{n})}
\hspace{-0.050cm}=\hspace{+0.020cm} 
O(1 + \hspace{+0.020cm}t)^{-\;\!\alpha}
\hspace{-0.030cm}$
as well,
\hspace{-0.030cm}provided that $\eqref{eqn_2_f}\hspace{-0.040cm}$ 
is satisfied.
\end{Theorem}
%
%

This result is related to the celebrated 
{\small \sc Wiegner's theorem}\;\cite{Wiegner1987},
recalled next. 
%
%
%
%
%
%
\begin{Theorem}[{\sc \small M.\,Wiegner,\,1987}]\label{Wiegner_thm}
\!Let
$\hspace{+0.030cm} \mbox{\boldmath $f$}(\cdot,t) 
\in L^{1}(\hspace{+0.010cm}
(\hspace{+0.020cm} 0, \infty), L^{2}_{\sigma}(\mathbb{R}^{n})
\hspace{+0.020cm}) \hspace{-0.020cm}$
satisfy $\hspace{+0.020cm}\eqref{eqn_2_f} \hspace{-0.035cm}$ 
for some  
$\hspace{+0.030cm} 0 \leq \alpha \leq (n+2)/4 $.
\!If
$\hspace{+0.020cm}\|\;\!\mbox{\boldmath $v$}(\cdot,t)\;\!
\|_{\scriptstyle L^{2}(\mathbb{R}^{n})}
\!\!\;\!\;\!=\:\! O(1 + \hspace{+0.020cm} t)^{-\;\!\beta} $
for some $\:\! 0 \leq \beta \leq \alpha $,
then
we will also have
$\hspace{+0.030cm}\|\hspace{+0.020cm}\mbox{\boldmath $u$}(\cdot,t)\;\!
\|_{\scriptstyle L^{2}(\mathbb{R}^{n})}
\hspace{-0.050cm}=\hspace{+0.020cm} 
O(1 + \hspace{+0.020cm}t)^{-\;\!\beta} \hspace{-0.040cm} $
and,
in addition\/\mbox{\em :} \\
\mbox{} \vspace{-0.700cm} \\
\begin{equation}\label{eqn_Wiegner_estimate}
\|\hspace{+0.035cm} \mbox{\boldmath $u$}(\cdot,t) 
-\hspace{+0.027cm}
\mbox{\boldmath $v$}(\cdot,t) \hspace{+0.030cm}
\|_{\mbox{}_{\scriptstyle \!\;\!L^{2}(\mathbb{R}^{n})}}
\hspace{-0.040cm}=\hspace{+0.050cm}
\left\{\,
\begin{array}{ll}
\!o\hspace{+0.030cm}
(1 + \hspace{+0.020cm} t)^{-\;\!(n\:\!-\:\!2)/4} & 
\mbox{ if \;} \beta = 0 \\
\mbox{} \vspace{-0.250cm} \\
\!O\hspace{+0.010cm}
(1 + \hspace{+0.020cm} t)^{-\;\!2\;\!\beta\,-\;\!(n\:\!-\:\!2)/4} & 
\mbox{ if \;} 0 < \beta < \frac{1}{2} \\
\mbox{} \vspace{-0.250cm} \\
\!O\hspace{+0.010cm}
(1 + \hspace{+0.020cm}t)^{-\;\!(n\:\!+\:\!2)/4} \;\!
\log\;\!(2 + \hspace{+0.010cm}t) & 
\mbox{ if \;} \beta = \frac{1}{2} \\
\mbox{} \vspace{-0.250cm} \\
\!O\hspace{+0.010cm}
(1 + \hspace{+0.020cm}t)^{-\;\!(n\:\!+\:\!2)/4} & 
\mbox{ if \;} \frac{1}{2} < \beta \leq \frac{n\:\!+\:\!2}{4} 
\end{array}
\right.
\end{equation}

\end{Theorem}	
\mbox{} \vspace{-0.700cm} \\
%
%
\begin{proof}[Proof of Theorem\:\ref{Wiegner_thm}]
See\;\cite{Wiegner1987}, pp.\:305\;\!-\:\!311. 
\hspace{-0.030cm}For an alternative proof
when $ n \hspace{-0.020cm} \leq 3 $,
see \cite{KreissHagstromLorenzZingano2003, Zhou2007}. 
\end{proof}
%
%
\mbox{} \vspace{-1.100cm} \\
%
%
\begin{proof}[Proof of Theorem\:\ref{Inverse_Wiegner}]
We begin by considering
the case of dimension $ n \geq 3 $.
%
%
%
%
Let
$ \mbox{\boldmath $\theta$}\hspace{+0.010cm}(\cdot,t) 
\hspace{-0.030cm}= $
$ \mbox{\boldmath $u$}\hspace{+0.005cm}(\cdot,t) 
\hspace{-0.010cm}-\hspace{+0.020cm} 
\mbox{\boldmath $v$}\hspace{+0.007cm}(\cdot,t) $
and let\hspace{+0.020cm}
$ \| \hspace{+0.030cm} \cdot \hspace{+0.030cm} \| $
stand for
$ \| \hspace{+0.030cm} \cdot \hspace{+0.030cm} 
\|_{L^{2}(\mathbb{R}^{n})} \hspace{-0.020cm}$.
Applying~\eqref{eqn_Wiegner_estimate} 
of {\small \sc Theorem\;\ref{Wiegner_thm}} 
with 
$ \beta = 0 $,
we obtain that
$ \|\hspace{+0.040cm}\mbox{\boldmath $\theta$}
\hspace{+0.010cm}(\cdot,t) \hspace{+0.030cm}\| 
\hspace{-0.020cm}=\hspace{+0.020cm} 
o\hspace{+0.030cm}
(1 + \hspace{+0.015cm}t)^{-\hspace{+0.050cm}
(n\;\!-\;\!2)\hspace{+0.010cm}/\hspace{+0.010cm}4} \hspace{-0.040cm}$,
which gives the result 
if $\hspace{+0.020cm} \alpha \leq (n-2)/4 $.
For larger $ \alpha $,
this already gives us
$ \|\hspace{+0.030cm}\mbox{\boldmath $v$}(\cdot,t) \hspace{+0.030cm}\| 
\hspace{-0.020cm}=\hspace{+0.020cm} 
O\hspace{+0.010cm}
(1 + \hspace{+0.015cm}t)^{-\hspace{+0.050cm}(n\;\!-\;\!2)
\hspace{+0.010cm}/\hspace{+0.010cm}4} \hspace{-0.040cm}$,
and we proceed as follows.
If $ n > 4 $, \eqref{eqn_Wiegner_estimate} with $ \beta = (n - 2)/4 $
gives 
$ \|\hspace{+0.030cm}\mbox{\boldmath $\theta$}(\cdot,t) \hspace{+0.030cm}\| 
\hspace{-0.020cm}=\hspace{+0.020cm} 
O\hspace{+0.010cm}(1 + \hspace{+0.015cm}t)^{-\;\!(n\:\!+\:\!2)/4} \hspace{-0.040cm}$,
and so we are done. \linebreak
If $ n \hspace{-0.030cm}=\hspace{-0.030cm} 4 $,
\eqref{eqn_Wiegner_estimate} gives
$ \|\hspace{+0.030cm}\mbox{\boldmath $\theta$}(\cdot,t) \hspace{+0.030cm}\| 
\hspace{-0.020cm}=\hspace{+0.020cm} 
O(1 + \hspace{+0.015cm}t\hspace{+0.010cm})^{-\;\!(n\:\!+\:\!2)/4} 
\hspace{+0.030cm}\log \hspace{+0.060cm}(2 + t) 
\hspace{-0.020cm}=\hspace{+0.020cm} 
O(1 + \hspace{+0.015cm}t\hspace{+0.010cm})^{-\;\!1} 
\hspace{-0.030cm} $, \linebreak
so that
$ \|\hspace{+0.030cm}\mbox{\boldmath $v$}(\cdot,t) \hspace{+0.030cm}\| 
\hspace{-0.020cm}=\hspace{+0.020cm} 
O(1 + \hspace{+0.015cm}t)^{-\;\!\gamma} \hspace{-0.020cm} $
with
$ \gamma = \min \hspace{+0.050cm}
\{\hspace{+0.005cm}1, \hspace{+0.030cm} 
\alpha \hspace{+0.025cm} \} > 1/2 $.
This then \linebreak
gives, 
using \eqref{eqn_Wiegner_estimate} again, that 
$ \|\hspace{+0.030cm}\mbox{\boldmath $\theta$}(\cdot,t) \hspace{+0.030cm}\| 
\hspace{-0.020cm}=\hspace{+0.020cm} 
O(1 + \hspace{+0.015cm}t)^{-\;\!(n\:\!+\:\!2)/4} \hspace{-0.020cm}$,
which implies
$ \|\hspace{+0.030cm}\mbox{\boldmath $v$}(\cdot,t) \hspace{+0.030cm}\| 
\hspace{-0.020cm}=\hspace{+0.020cm} 
O(1 + \hspace{+0.015cm}t)^{-\;\!\alpha} \hspace{-0.040cm}$,
as claimed.
Finally,
in the remaining case
$\hspace{+0.015cm} n \hspace{-0.020cm}= 3 $,
$ \|\hspace{+0.030cm}\mbox{\boldmath $v$}(\cdot,t) \hspace{+0.030cm}\| 
\hspace{-0.020cm}=\hspace{+0.020cm} 
O(1 + \hspace{+0.015cm}t)^{-\;\!(n\:\!-\:\!2)/4} 
\hspace{-0.020cm}=\hspace{+0.020cm}
O(1 + \hspace{+0.015cm}t)^{-\;\!1/4} $
gives, applying \eqref{eqn_Wiegner_estimate},
$ \|\hspace{+0.030cm}\mbox{\boldmath $\theta$}(\cdot,t) \hspace{+0.030cm}\| 
\hspace{-0.020cm}=\hspace{+0.020cm} 
O(1 + \hspace{+0.015cm}t)^{-\;\!3/4} $,
so that 
$ \|\hspace{+0.030cm}\mbox{\boldmath $v$}(\cdot,t) \hspace{+0.030cm}\| 
\hspace{-0.020cm}=\hspace{+0.020cm} 
O(1 + \hspace{+0.015cm}t)^{-\;\!\gamma} $
where 
$ \gamma = \min\;\!\{\hspace{+0.010cm} \alpha, 3/4 \hspace{+0.010cm}\}$.
\linebreak
If $ \alpha \leq 3/4 $, we are done;
otherwise, 
$ \|\hspace{+0.030cm}\mbox{\boldmath $v$}(\cdot,t) \hspace{+0.030cm}\| 
\hspace{-0.020cm}=\hspace{+0.020cm} 
O(1 + \hspace{+0.015cm}t)^{-\;\!3/4} \hspace{-0.020cm}$
and so, by \eqref{eqn_Wiegner_estimate} again, 
we have
$ \|\hspace{+0.030cm}\mbox{\boldmath $\theta$}(\cdot,t) \hspace{+0.030cm}\| 
\hspace{-0.030cm}=\hspace{+0.020cm} 
O(1 + \hspace{+0.015cm}t)^{-\;\!(n\:\!+\:\!2)/4} \hspace{-0.050cm} $,
giving that
$ \|\hspace{+0.030cm}\mbox{\boldmath $v$}(\cdot,t) \hspace{+0.030cm}\| 
\hspace{-0.030cm}=\hspace{+0.020cm} 
O(1 + \hspace{+0.015cm}t)^{-\;\!\alpha} \hspace{-0.030cm} $,
as claimed. 
This completes the proof of
{\small \sc Theorem\;\ref{Inverse_Wiegner}}
for any $\hspace{+0.010cm} n \!\;\!\geq 3 $,
as a
direct consequence of \eqref{eqn_Wiegner_estimate} above.
\mbox{} \vspace{-0.450cm} \\
%
%
%
%

In the remaining case $ n = 2 $,
letting once more
$ \mbox{\boldmath $\theta$}(\cdot,t) 
\hspace{-0.030cm}= 
\mbox{\boldmath $u$}(\cdot,t) 
\hspace{-0.010cm}-\hspace{+0.020cm} 
\mbox{\boldmath $v$}(\cdot,t) $,
we have
$ \theta(\cdot,t) \in 
C([\hspace{+0.030cm}0, \infty), L^{2}(\mathbb{R}^{2})) $
and \\
\mbox{} \vspace{-0.700cm} \\
\begin{equation}\label{eqn_theta}
\mbox{\boldmath $\theta$}_t \hspace{+0.030cm}=\, 
\Delta \hspace{+0.020cm} \mbox{\boldmath $\theta$}
\;\!-\;\!
\mathbb{P}_{\!\:\!H} \:\![\,
\mbox{\boldmath $u$}\hspace{+0.010cm}(\cdot,t) \cdot 
\nabla \mbox{\boldmath $u$}\hspace{+0.010cm}(\cdot,t) \,],
\quad \;\;\,
\mbox{\boldmath $\theta$}(\cdot,0) = \mbox{\bf 0},
\end{equation}
\mbox{} \vspace{-0.250cm} \\
where $ \;\!\mathbb{P}_{\!\;\!H} \!: 
L^{2}(\mathbb{R}^{2})
\hspace{-0.030cm}\rightarrow\hspace{-0.030cm} 
L^{2}_{\sigma}(\mathbb{R}^{2}) \hspace{+0.020cm}$
denotes the\hspace{+0.030cm} 
{\small \sc  Leray-Helmholtz projector}
(see e.g.\;\cite{RobinsonSadowski2016},\; Chapter 2).  
Our assumption in this case is that,
for some $ 0 < \alpha \leq (n + 2)/4 = 1 $, \\
\mbox{} \vspace{-0.700cm} \\
\begin{equation}\label{eqn_decay_u_f}
\|\;\!\mbox{\boldmath $u$}(\cdot,t) 
\hspace{+0.030cm} \| 
\,=\:
O\hspace{+0.010cm}(\hspace{+0.010cm} 1 + \hspace{+0.020cm}t
\hspace{+0.010cm})^{-\;\!\alpha}
\quad \mbox{and } \quad
\|\;\!\mbox{\boldmath $f$}(\cdot,t) 
\hspace{+0.030cm} \| 
\,=\:
O\hspace{+0.010cm}(\hspace{+0.010cm} 1 + \hspace{+0.020cm}t
\hspace{+0.010cm})^{-\;\!\alpha\;\!-\:\!1}
\end{equation}
\mbox{} \vspace{-0.300cm} \\
for all $\hspace{+0.020cm} t \!\;\!> 0 $,
where\hspace{+0.020cm}
$ \| \hspace{+0.030cm} \cdot \hspace{+0.030cm} \| $
denotes 
the norm
$ \| \hspace{+0.030cm} \cdot \hspace{+0.030cm} 
\|_{L^{2}(\mathbb{R}^{2})} \hspace{-0.020cm}$.
Now, \eqref{eqn_decay_u_f} gives \\
\mbox{} \vspace{-0.700cm} \\
\begin{equation}\label{eqn_decay_gradient} 
\|\hspace{+0.030cm}D\hspace{+0.010cm}
\mbox{\boldmath $u$}(\cdot,t) \hspace{+0.030cm} \| 
=\hspace{+0.030cm}
O\hspace{+0.010cm}(\hspace{+0.010cm} 1 + \hspace{+0.020cm}t
\hspace{+0.015cm})^{-\;\!\alpha} \hspace{+0.080cm}
t^{\:\!-\;\!1/2} 
\end{equation}
\mbox{} \vspace{-0.300cm} \\
(see \cite{GuterresNichePerusatoZingano2021}
for a broader discussion). 
%
%
In fact, 
recalling the energy identity \\
\mbox{} \vspace{-0.600cm} \\
\begin{equation}
\notag
\frac{\dd}{\dd t} \, 
\|\hspace{+0.020cm} D\hspace{+0.010cm}
\mbox{\boldmath $u$}(\cdot,t)\hspace{+0.020cm} \|^{\:\!2}
+\;\!
2 \;\!\;\!
\|\hspace{+0.020cm} D^{2}\hspace{+0.010cm}
\mbox{\boldmath $u$}(\cdot,t)\hspace{+0.020cm} \|^{\:\!2}
\hspace{+0.020cm}=\hspace{+0.090cm} 
2 \hspace{-0.050cm}
\int_{\mbox{}_{\scriptstyle \!\;\!\mathbb{R}^{2}}} \!\!
\langle \,\nabla \mbox{\footnotesize $\wedge$}\;\!(\hspace{+0.020cm}
\nabla \mbox{\footnotesize $\wedge$} \;\!\mbox{\boldmath $u$}), 
\hspace{+0.020cm} \mbox{\boldmath $f$}(x,t) \, \rangle
\dd x
\end{equation}
\mbox{} \vspace{-0.175cm} \\
(for almost all
$\!\;\!\;\!t > 0 $),
we get \\
\mbox{} \vspace{-0.650cm} \\
\begin{equation}
\notag
\frac{\dd}{\dd t} \, 
\|\hspace{+0.020cm} D\hspace{+0.010cm}
\mbox{\boldmath $u$}(\cdot,t)\hspace{+0.020cm} \|^{\:\!2}
\hspace{+0.030cm}\leq\hspace{+0.070cm}
\mbox{\small $ {\displaystyle \frac{1}{2} }$} 
\;
\|\;\!\mbox{\boldmath $f$}(\cdot,t)\hspace{+0.020cm} \|^{\:\!2}
\end{equation}
\mbox{} \vspace{-0.200cm} \\
for a.e.\;$t > 0 $.
By \eqref{eqn_decay_u_f} above,
we have, 
for some constant 
$\hspace{+0.020cm} 
C_{\hspace{-0.030cm}f} \hspace{-0.050cm}> 0 \hspace{+0.010cm}$,
that
$ {\displaystyle
\|\;\!\mbox{\boldmath $f$}(\cdot,t) 
\hspace{+0.030cm} \| 
\hspace{+0.010cm}\leq\hspace{+0.020cm}
C_{\hspace{-0.030cm}f} \hspace{+0.025cm}
(\hspace{+0.010cm} 1 + \hspace{+0.020cm}t
\hspace{+0.010cm})^{-\;\!\alpha\;\!-\:\!1}
} $
\hspace{-0.030cm}(cf.\;(\ref{eqn_2_f})), 
and so the the function
$ \!\;\!z_{1} \hspace{-0.030cm}\in\!\;\! C^{0}(0, \infty) $
\!\;\!given by \\
\mbox{} \vspace{-0.575cm} \\
\begin{equation}
\notag
z_{1}(t) :=\;
\|\hspace{+0.020cm} D\hspace{+0.010cm}
\mbox{\boldmath $u$}(\cdot,t)\hspace{+0.020cm} \|^{\:\!2}
\hspace{+0.030cm}+\hspace{+0.090cm}
\mbox{\small $ {\displaystyle
      \frac{1}{\;\!4\hspace{+0.020cm} \alpha + 2\;\!} }$} \,
C_{\!f}^{\;\!2} \;\!
(1 + t\hspace{+0.010cm})^{-\;\!2\;\!\alpha \;\!-\:\!1}
\end{equation}
\mbox{} \vspace{-0.175cm} \\
is monotonically decreasing in the interval
$(\hspace{+0.010cm} 0, \infty) $.
Because \\
\mbox{} \vspace{-0.600cm} \\
\begin{equation}
\notag
\|\;\!\mbox{\boldmath $u$}(\cdot, \:\!t)\hspace{+0.020cm} \|^{\:\!2}
\hspace{+0.040cm}+\hspace{+0.070cm}
2 \! \int_{\mbox{}_{\scriptstyle \!\:\!t/2}}^{t}
\hspace{-0.070cm}
\|\hspace{+0.020cm} D\hspace{+0.010cm}
\mbox{\boldmath $u$}(\cdot, \tau)\hspace{+0.020cm} \|^{\:\!2}
\dd\tau
\;=\;
\|\;\!\mbox{\boldmath $u$}(\cdot, \:\!t/2)\hspace{+0.020cm} \|^{\:\!2}
\end{equation}
\mbox{} \vspace{-0.100cm} \\
(for all $ t > 0 $),
we then have \\
\mbox{} \vspace{-0.200cm} \\
\mbox{} \hspace{+0.750cm} 
$ {\displaystyle
t\; \|\hspace{+0.020cm} D\hspace{+0.010cm}
\mbox{\boldmath $u$}(\cdot, \:\!t)\hspace{+0.020cm} \|^{\:\!2}
\hspace{+0.010cm}\leq\:
t \;\! z_{1}(t) 
\,\leq\;
2 \!
\int_{\mbox{}_{\scriptstyle \!\:\! t/2}}^{t}
\hspace{-0.150cm}
z_{1}(\tau) \dd \tau
} $ \\
\mbox{} \vspace{+0.150cm} \\
\mbox{} \hspace{+4.250cm}
$ {\displaystyle
=\;\:\! 2 \!
\int_{\mbox{}_{\scriptstyle \!\:\! t/2}}^{t}
\hspace{-0.100cm} 
\|\hspace{+0.020cm} D\hspace{+0.010cm}
\mbox{\boldmath $u$}(\cdot, \tau)\hspace{+0.020cm} \|^{\:\!2}
\dd\tau
\,+\,
\mbox{\small ${\displaystyle 
      \frac{C_{\!f}^{\;\!2}}{\;\!2\:\!\alpha + 1\;\!} }$}
\!
\int_{\mbox{}_{\scriptstyle \hspace{-0.050cm} t/2}}^{t}
\hspace{-0.100cm}
(1 + \tau)^{-\;\!2\:\!\alpha\;\!-\;\!1}
\dd\tau
} $ \\
\mbox{} \vspace{+0.120cm} \\
\mbox{} \hspace{+4.270cm}
$ {\displaystyle
\leq\;
\|\;\! \mbox{\boldmath $u$}(\cdot, t/2)\hspace{+0.020cm} \|^{\:\!2}
\,+\;\!\;\!
\mbox{\small ${\displaystyle 
	  \frac{C_{\!f}^{\;\!2}}{\;\!2\:\!\alpha + 1\;\!} }$}
\,
\mbox{\small $ {\displaystyle
	  \frac{2^{\:\!2\:\!\alpha}\;\!}{2 \:\! \alpha} }$} 
\, (1 + t\hspace{+0.010cm})^{-\;\!2\:\!\alpha}
} $ \\
\mbox{} \vspace{+0.100cm} \\
\mbox{} \hspace{+4.270cm}
$ {\displaystyle
\leq\;
2^{\;\!2\hspace{+0.010cm}\alpha} \hspace{+0.040cm} 
C_{\!\;\!0}^{\;\!2} \;\!
(1 + t\hspace{+0.010cm})^{-\;\!2\:\!\alpha}
\;\!+\;\!\;\!
\mbox{\small ${\displaystyle 
	  \frac{C_{\!f}^{\;\!2}}{\;\!2\:\!\alpha + 1\;\!} }$}
\,
\mbox{\small $ {\displaystyle
	  \frac{2^{\:\!2\:\!\alpha}\;\!}{2 \:\! \alpha} }$} 
\, (1 + t\hspace{+0.010cm})^{-\;\!2\:\!\alpha}
} $ \\
\mbox{} \vspace{+0.050cm} \\
where in the last step
we have used that
\hspace{-0.030cm}(by \eqref{eqn_decay_u_f} above): 
$ \!\!\;\!\;\!\|\;\!\mbox{\boldmath $u$}(\cdot,t) \hspace{+0.020cm}
\|_{L^{2}(\mathbb{R}^{2})} \hspace{-0.020cm}\leq\!\;\! 
C_{\!\;\!0} \hspace{+0.020cm}
(1 +\hspace{+0.020cm} t)^{-\;\!\alpha} $
for some constant
$C_{\!\;\!0} \!\;\!> 0 $.
This produces \eqref{eqn_decay_gradient}, as claimed.
%
%
Now, from \eqref{eqn_theta} 
we have \\
\mbox{} \vspace{-0.010cm} \\
\mbox{} \hspace{-0.1250cm}
$ {\displaystyle
\mbox{\boldmath $\theta$}(\cdot,t)
\hspace{+0.020cm}=\;\!-\hspace{-0.020cm}
\int_{\mbox{}_{\scriptstyle \!\;\!0}}^{t}
\!\!\;\!
e^{\;\!\mbox{\footnotesize $ {\displaystyle
	   \Delta (t - s) }$}}
\hspace{+0.040cm} \bigl[\,
\mathbb{P}_{\!\;\!H} \hspace{+0.020cm} [\, \mbox{\boldmath $u$} 
\!\;\!\cdot\!\;\!\nabla \mbox{\boldmath $u$} \hspace{+0.042cm}]
\hspace{+0.010cm}(\cdot,s)
\hspace{+0.030cm}\bigr]
\dd s
=\;\!-\hspace{-0.020cm}
\int_{\mbox{}_{\scriptstyle \!\;\!0}}^{t}
\hspace{-0.070cm}
\mathbb{P}_{\!\;\!H} \hspace{+0.030cm} \bigl[\,
e^{\;\!\mbox{\footnotesize $ {\displaystyle
       \Delta (t - s) }$}}
\hspace{+0.040cm} 
[\, \mbox{\boldmath $u$} 
\!\;\!\cdot\!\;\!\nabla \mbox{\boldmath $u$} ]
\bigr]
\dd s
} $,\\
\mbox{} \vspace{+0.020cm} \\
for every $\hspace{+0.010cm} t \!\;\!> 0 $,
which gives \\
\mbox{} \vspace{-0.650cm} \\
\begin{equation}
\notag
\|\,\mbox{\boldmath $\theta$}(\cdot,t) \hspace{+0.030cm} \|
\hspace{+0.070cm} \leq
\int_{\mbox{}_{\scriptstyle \!\;\!0}}^{t}
\hspace{-0.050cm} \|\,
e^{\;\!\mbox{\footnotesize $ {\displaystyle
	   \Delta (t - s) }$}}
\hspace{+0.040cm} 
[\, \mbox{\boldmath $u$} 
\!\;\!\cdot\!\;\!\nabla \mbox{\boldmath $u$} \,](\cdot,s)
\;\!\|
\dd s
\end{equation}
\mbox{} \vspace{-0.700cm} \\
\begin{equation}
\notag
=
\int_{\mbox{}_{\scriptstyle \!\;\!0}}^{\hspace{-0.010cm}t/2}
\hspace{-0.070cm} \|\,
e^{\;\!\mbox{\footnotesize $ {\displaystyle
       \Delta (t - s) }$}}
\hspace{+0.040cm} 
[\, \mbox{\boldmath $u$} 
\!\;\!\cdot\!\;\!\nabla \mbox{\boldmath $u$} \,](\cdot,s)
\;\!\|
\dd s
\;+
\int_{\mbox{}_{\scriptstyle \!\;\!t/2}}^{\:\!t}
\hspace{-0.070cm}\|\,
e^{\;\!\mbox{\footnotesize $ {\displaystyle
	   \Delta (t - s) }$}}
\hspace{+0.040cm} 
[\, \mbox{\boldmath $u$} 
\!\;\!\cdot\!\;\!\nabla \mbox{\boldmath $u$} \,](\cdot,s)
\;\!\|
\dd s
\end{equation}
\mbox{} \vspace{-0.650cm} \\
\begin{equation}
\notag
\leq
\int_{\mbox{}_{\scriptstyle \!\;\!0}}^{t/2}
\hspace{-0.100cm} (t - s)^{-\;\!1} \,
\|\;\!\mbox{\boldmath $u$}(\cdot,s)\;\! \|^{\:\!2}
\dd s
\hspace{+0.100cm} + \hspace{+0.025cm}
\int_{\mbox{}_{\scriptstyle \!t/2}}^{t}
\hspace{-0.050cm} 
(t - s)^{-\;\!1/2} \,
\|\;\!\mbox{\boldmath $u$}(\cdot,s)\;\! \|\;
\hspace{+0.020cm}
\|\hspace{+0.020cm} D \hspace{+0.010cm} 
\mbox{\boldmath $u$}(\cdot,s)\;\! \|
\dd s
\end{equation}
\mbox{} \vspace{-0.075cm} \\
using standard properties of the heat kernel,
see e.g.\;(\cite{KreissHagstromLorenzZingano2003}, p.\hspace{+0.100cm}236)
or (\cite{Zhou2007}, p.\,1227).
By \eqref{eqn_decay_u_f} and \eqref{eqn_decay_gradient},
we then obtain \\
\mbox{} \vspace{-0.700cm} \\
\begin{equation}
\notag
\|\;\!\mbox{\boldmath $\theta$}(\cdot,t) \hspace{+0.030cm} \|
\;\!\;\!=\;
O\hspace{+0.010cm}(\hspace{+0.020cm} t^{\;\!-\;\!1}) 
\hspace{-0.070cm} 
\int_{\mbox{}_{\scriptstyle \!\;\!0}}^{t/2}
\hspace{-0.150cm}  
(1 + s)^{-\;\!2\hspace{+0.010cm}\alpha} 
\hspace{+0.040cm} ds
\hspace{+0.100cm}+\;
O\hspace{+0.0110cm}(\hspace{+0.020cm} 
t^{\;\!-\;\!2\hspace{+0.010cm} \alpha}).
\end{equation}
\mbox{} \vspace{-0.170cm} \\
This gives
$ \|\hspace{+0.040cm}
\mbox{\boldmath $\theta$}(\cdot,t)
\hspace{+0.030cm} \| 
\hspace{-0.020cm}=\hspace{+0.020cm} 
O(\hspace{+0.010cm}t^{\:\!-\;\!2 \hspace{+0.020cm} \alpha} 
\hspace{-0.010cm}) $
if $\hspace{+0.010cm} \alpha < 1/2 $,
$ \|\hspace{+0.040cm}
\mbox{\boldmath $\theta$}(\cdot,t)
\hspace{+0.030cm} \| 
\hspace{-0.020cm}=\hspace{+0.020cm} 
O(\hspace{+0.010cm} t^{\:\!-\;\!1}) \hspace{+0.000cm}
\log\;\!(2 + t) $
if $\hspace{+0.010cm} \alpha \hspace{-0.020cm}= 1/2 $,
and
$ \|\hspace{+0.040cm}
\mbox{\boldmath $\theta$}(\cdot,t)
\hspace{+0.030cm} \| 
\hspace{-0.020cm}=\hspace{+0.020cm} 
O(\hspace{+0.010cm}t^{\:\!-\;\!1}) $
if 
$ 1/2 < \alpha \leq \hspace{-0.010cm} (n+2)/4 = 1 $.
Since
$ \mbox{\boldmath $\theta$} =\hspace{+0.010cm}
\mbox{\boldmath $u$} -
\mbox{\boldmath $v$}$, \linebreak
it then follows from
$ \|\hspace{+0.040cm}
\mbox{\boldmath $u$}(\cdot,t)
\hspace{+0.030cm} \| 
\hspace{-0.020cm}=\hspace{+0.020cm} 
O(\hspace{+0.010cm}t^{\:\!-\;\!\alpha}) $,
cf.\;\eqref{eqn_decay_u_f},
that
$ \|\hspace{+0.040cm}
\mbox{\boldmath $v$}(\cdot,t)
\hspace{+0.030cm} \| 
\hspace{-0.020cm}=\hspace{+0.020cm} 
O(\hspace{+0.010cm}t^{\:\!-\;\!\alpha}) $.
This concludes the proof of
{\small \sc Theorem\:\ref{Inverse_Wiegner}}.
\end{proof}
%
%
%
\mbox{} \vspace{-1.050cm} \\
%
%
%
%
\begin{Remark}\label{remark1}
Let $\hspace{+0.010cm} 2 \leq n \leq 4 $,
$ \mbox{\boldmath $u$}_0 \hspace{-0.040cm} \in
L^{2}_{\sigma}(\mathbb{R}^{n}) $.
Given
$ \hspace{-0.015cm}\mbox{\boldmath $f$}(\cdot,t) \hspace{-0.015cm} 
\in\hspace{-0.020cm} L^{1}((0,\infty), L^{2}_{\sigma}(\mathbb{R}^{n})) $
satisfying \eqref{eqn_2_f} and,
in addition,
$ \|\hspace{+0.030cm} \mbox{\boldmath $f$}(\cdot,t)
\hspace{+0.020cm} \|_{\dot{H}^{m}} \!= 
O(1 + t)^{-\;\!\alpha\;\!-\:\!1\:\!-\;\! m/2} $
for some $ m \geq 0 $,
then 
it follows from Wiegner's theorems above
and\hspace{+0.020cm} 
(\hspace{+0.005cm}\cite{GuterresNichePerusatoZingano2021},\!
{\small \sc Theorem\;\!\;\!1.1})
that, for $ t \gg 1 $:
$ \|\hspace{+0.005cm} D^{\ell}\hspace{+0.005cm}
\mbox{\boldmath $u$}\hspace{+0.005cm}(\cdot,t)\hspace{+0.020cm}
\|_{L^{2}(\mathbb{R}^{n})} \hspace{-0.030cm}= 
\hspace{+0.010cm}O\hspace{+0.010cm}(\hspace{+0.010cm}
t^{\:\!-\;\!\alpha 
\hspace{+0.050cm}-\hspace{+0.050cm} 
\ell\hspace{+0.010cm}/2}\hspace{+0.010cm}) \hspace{-0.010cm}$
and \\ 
\mbox{} \vspace{-0.550cm} \\
\begin{equation}
\notag
\|\hspace{+0.040cm}\mbox{\boldmath $u$}(\cdot,t) -
\mbox{\boldmath $v$}(\cdot,t) \hspace{+0.020cm}
\|_{\mbox{}_{\scriptstyle \dot{H}^{\ell}(\mathbb{R}^{n})}}
\hspace{-0.050cm}=\hspace{+0.090cm}
\left\{\hspace{-0.020cm}
\begin{array}{ll}
\!o\hspace{+0.030cm}(\hspace{+0.010cm}t^{\:\!-\;\!
(n \hspace{+0.020cm}-\hspace{+0.020cm}2)\hspace{+0.010cm}/\hspace{+0.010cm}4
\hspace{+0.050cm}-\hspace{+0.050cm} \ell\hspace{+0.010cm}/\hspace{+0.010cm}2}
\hspace{+0.005cm}) & 
\mbox{ if \;} \alpha = 0 \\
\mbox{} \vspace{-0.250cm} \\
\!O\hspace{+0.010cm}(\hspace{+0.010cm}t^{\:\!-\;\!
2\;\!\alpha\,-\,(n \hspace{+0.020cm}-\hspace{+0.020cm}2)
\hspace{+0.010cm}/\hspace{+0.010cm}4  
\hspace{+0.050cm}-\hspace{+0.050cm} 
\ell\hspace{+0.010cm}/\hspace{+0.010cm}2}
\hspace{+0.005cm}) & 
\mbox{ if \;} 0 < \alpha < \frac{1}{2} \\
\mbox{} \vspace{-0.250cm} \\
\!O\hspace{+0.010cm}(\hspace{+0.010cm}t^{\:\!-\;\!
(n \hspace{+0.020cm}+\hspace{+0.020cm}2)\hspace{+0.010cm}/\hspace{+0.010cm}4
\hspace{+0.050cm}-\hspace{+0.050cm} \ell\hspace{+0.010cm}/\hspace{+0.010cm}2}
\hspace{+0.005cm}) \hspace{+0.005cm} \log\hspace{+0.020cm}t & 
\mbox{ if \;} \alpha = \frac{1}{2} \\
\mbox{} \vspace{-0.250cm} \\
\!O\hspace{+0.010cm}(\hspace{+0.010cm}t^{\:\!-\;\!
(n \hspace{+0.020cm}+\hspace{+0.020cm}2)\hspace{+0.010cm}/\hspace{+0.010cm}4
\hspace{+0.050cm}-\hspace{+0.010cm} \ell\hspace{+0.010cm}/\hspace{+0.010cm}2} 
\hspace{+0.005cm}) & 
\mbox{ if \;} \frac{1}{2} < \alpha \leq
\frac{n\hspace{+0.010cm}+\hspace{+0.020cm}2}{4} 
\end{array}
\right.
\end{equation}
\mbox{} \vspace{+0.050cm} \\
for every $\hspace{+0.010cm} 0 \leq \ell \leq m + 1 $,
provided that we have
$ \|\hspace{+0.040cm} \mbox{\boldmath $u$}(\cdot,t)
\hspace{+0.020cm} \|_{L^{2}(\mathbb{R}^{n})} \!=
O\hspace{+0.010cm}(1 + t)^{-\;\!\alpha} $
or 
$ \|\hspace{+0.040cm} \mbox{\boldmath $v$}(\cdot,t)
\hspace{+0.020cm} \|_{L^{2}(\mathbb{R}^{n})} \!=
O\hspace{+0.010cm}(1 + t)^{-\;\!\alpha} \!\;\!$. 
\hspace{-0.040cm}As before,
$ \mbox{\boldmath $v$}(\cdot,t) $
denotes the solution of~{\hspace{-0.030cm}\eqref{eqn_stokes_flow}}, 
$ \hspace{-0.040cm} \dot{H}^{\ell}(\mathbb{R}^{n}) $
stands for the homogeneous Sobolev space of order~$\ell$
$($\cite{Chemin2011}, p.\,$25$\hspace{+0.005cm}$)$, 
and
$ \|\hspace{+0.040cm} \mbox{\bf w} \hspace{+0.040cm}
\|_{\dot{H}^{\ell}(\mathbb{R}^{n})} \hspace{-0.070cm}=
\|\hspace{+0.030cm} D^{\ell}\hspace{+0.020cm} 
\mbox{\bf w} \hspace{+0.040cm} 
\|_{L^{2}(\mathbb{R}^{n})} $. 
See also \cite{OliverTiti2000, SchonbekWiegner}.
\end{Remark}
\mbox{} \vspace{-1.050cm} \\
%
%
%
%
\begin{Remark}\label{remark2}
\mbox{} \hspace{-0.150cm}Under the same conditions given in 
\hspace{+0.010cm}{\small \sc Remark\;\!\;\!\ref{remark1}},
\hspace{-0.040cm}but assuming more strongly 
that 
$ \|\hspace{+0.010cm} \mbox{\boldmath $f$}(\cdot,t) 
\hspace{+0.020cm}\|_{L^{2}(\mathbb{R}^{n})} 
\hspace{-0.040cm}=\hspace{+0.010cm}
o\hspace{+0.020cm}(1 + t)^{-\;\!\alpha \;\!-\:\!1} \hspace{-0.050cm}$,
it follows from\,Wiegner's theorems 
and $($\hspace{+0.005cm}\cite{GuterresNichePerusatoZingano2021},\!
{\small \sc Theorem\,1.5}\hspace{+0.010cm}$)$
\hspace{-0.040cm}that \\
\mbox{} \vspace{-0.700cm} \\
\begin{equation}
\notag
\liminf_{t\,\rightarrow\,\infty} \hspace{+0.050cm}
t^{\:\!\alpha \hspace{+0.050cm}+\hspace{+0.050cm}
\ell\hspace{+0.010cm}/2} \;\!
\|\hspace{+0.020cm}D^{\ell} 
\mbox{\boldmath $u$}\hspace{+0.010cm}(\cdot,t) 
\hspace{+0.020cm}
\|_{\mbox{}_{\scriptstyle L^{2}(\mathbb{R}^{n})}}
\hspace{-0.020cm}>\hspace{+0.050cm}0
\end{equation}
\mbox{} \vspace{-0.225cm} \\
for all
$\hspace{+0.030cm} 0 \hspace{-0.010cm}\leq\hspace{-0.015cm} 
\ell \hspace{-0.020cm}\leq\hspace{-0.015cm} m + 1$,
provided that one has
$ {\displaystyle 
\lambda(\alpha) \hspace{-0.030cm}\equiv\hspace{+0.010cm}
\liminf_{t\,\rightarrow\,\infty} \hspace{+0.030cm}
t^{\:\!\alpha} \hspace{+0.020cm}
\|\hspace{+0.050cm}\mbox{\boldmath $u$}(\cdot,t) \hspace{+0.020cm}
\|_{\scriptstyle L^{2}(\mathbb{R}^{n})}
\hspace{-0.060cm}> 0 \hspace{+0.030cm}
} $ 
or else that
$\hspace{+0.020cm} 0 \!\;\!\leq\!\;\! 
\alpha \!\;\!<\!\;\! (n + 2)/4 \hspace{+0.030cm} $
and 
$ {\displaystyle 
\hspace{+0.050cm}
\liminf_{t\,\rightarrow\,\infty} \hspace{+0.030cm}
t^{\:\!\alpha} \hspace{+0.020cm}
\|\hspace{+0.050cm}\mbox{\boldmath $v$}(\cdot,t) \hspace{+0.020cm}
\|_{\scriptstyle L^{2}(\mathbb{R}^{n})}
\hspace{-0.040cm} > 0
} $,
where 
$ \mbox{\boldmath $v$}(\cdot,t) $
is the Stokes flow given in (\ref{eqn_stokes_flow})
above.
Moreover
$($cf.\hspace{+0.070cm}\cite{GuterresNichePerusatoZingano2021},%
\;\!{\small \sc Theorem\,1.5}\hspace{+0.010cm}$)$,  \\
\mbox{} \vspace{-0.575cm} \\
\begin{equation}
\notag
\liminf_{t\,\rightarrow\,\infty} \hspace{+0.060cm}
t^{\:\!\alpha \hspace{+0.040cm}+\hspace{+0.040cm}
\ell/2} \hspace{+0.050cm}
\|\hspace{+0.020cm} D^{\ell}
\mbox{\boldmath $u$}(\cdot,t) \hspace{+0.020cm}
\|_{\mbox{}_{\scriptstyle L^{2}(\mathbb{R}^{n})}}
\hspace{-0.040cm}\geq\hspace{+0.010cm}
K\hspace{-0.020cm}(\alpha,\ell,r) 
\cdot \hspace{+0.020cm}
\liminf_{t\,\rightarrow\,\infty} \hspace{+0.060cm}
t^{\:\!\alpha} \hspace{+0.030cm}
\|\hspace{+0.040cm}\mbox{\boldmath $u$}(\cdot,t) \hspace{+0.020cm}
\|_{\mbox{}_{\scriptstyle L^{2}(\mathbb{R}^{n})}}
\end{equation}
\mbox{} \vspace{-0.175cm} \\
where 
$ {\displaystyle
\hspace{+0.020cm}
r = \lambda(\alpha)/
\limsup_{t\,\rightarrow\,\infty} \hspace{+0.050cm}
t^{\:\!\alpha} \;\!
\|\hspace{+0.050cm}\mbox{\boldmath $u$}(\cdot,t) \hspace{+0.020cm}
\|_{\scriptstyle L^{2}(\mathbb{R}^{n})} 
} $, 
for some constant
$\hspace{-0.020cm}K\hspace{-0.020cm}(\alpha,\ell,r) 
\!\;\!> 0 \hspace{+0.020cm} $
which depends only on $ \alpha, \ell \hspace{+0.010cm}$
and $ r $,
and not on
$ \mbox{\boldmath $u$}_0\hspace{+0.010cm}
(\hspace{+0.010cm}\cdot\hspace{+0.010cm})$ 
or the solution $ \mbox{\boldmath $u$}(\cdot,t) $.
For other results related to this discussion,
see e.g.\hspace{+0.100cm}\cite{MR3493117, BrandoleseSchonbek2018, 
HagstromLorenzZinganoZingano2019, SchonbekWiegner}.
\end{Remark}
%
%
\mbox{} \vspace{-0.900cm} \\
%
%
%
%
%

\section{Applications}
\label{sec:dc}

\subsection{Two-sided algebraic estimates}

A first application of both Wiegner and inverse Wiegner theorem
in the case of the unforced Navier--Stokes equations,
is the following complete characterization of solutions satisfying two-sided algebraic decay estimates.

\begin{Corollary}
\label{cor:tsb}
Let $\hspace{+0.020cm}0<\alpha<(n+2)/2$, $\boldsymbol u_0\in L^2_\sigma(\R^n)$
\hspace{-0.030cm}and 
$\boldsymbol u$\hspace{+0.010cm}$(\cdot,t)$ 
be a Leray solution arising from $\boldsymbol u_0$, 
with $\boldsymbol f\equiv0$.
Then the two following properties are equivalent:
\begin{itemize}
\item[(i)] $(1+t)^{-\alpha}\lesssim \hspace{+0.030cm}
\|\hspace{+0.030cm}\boldsymbol u(\cdot,t)
\hspace{+0.030cm}\|^2_{L^2(\RR^n)} 
\hspace{+0.010cm}\lesssim \hspace{+0.030cm}
(1+t)^{-\alpha}$.
\item[(ii)] $(1+t)^{-\alpha}
\hspace{+0.010cm}\lesssim\hspace{+0.030cm}
\|\hspace{+0.050cm}e^{t\Delta}\boldsymbol u_0
\hspace{+0.020cm}\|^2_{L^2(\RR^n)}
\lesssim \hspace{+0.010cm}
(1+t)^{-\alpha}$.
\end{itemize}
\end{Corollary}

This is immediate.
The fact that (ii) implies (i) just relies 
on direct Wiegner's (Theorem~\ref{Wiegner_thm}).
The inverse Wiegner theorem is needed to prove that 
the upper bound in (i) implies the upper bound in (ii). 
Then one uses once more the direct Wiegner theorem to prove 
that the lower bound in (i) implies the lower bound in (ii).

The main interest of Corollary~\ref{cor:tsb} is that 
one can completely caracterize
the initial data such that (ii) holds.
\begin{Theorem}
\label{th:char}
Let $\alpha>0$, $u_0\in L^2(\RR^n)$. 
The following properties
are equivalent\hspace{+0.040cm}$:$
\begin{itemize}
\item[(i)] $(1+t)^{-\alpha}\lesssim\|e^{t\Delta}u_0\|^2_{L^2(\RR^n)}\lesssim (1+t)^{-\alpha}$,
\item[(ii)]
\,\,\,\,\,\,\mbox{$\displaystyle\liminf_{\rho\to0+}\rho^{-2\alpha}\int_{|\xi|\le \rho} |\widehat u_0(\xi)|^2\dd \xi>0$}\quad and\quad 
	\vspace{+0.15cm}
    \begin{sloppypar}
\mbox{$\displaystyle\limsup_{\rho\to0+} \rho^{-2\alpha}\int_{|\xi|\le \rho}|\widehat u_0(\xi)|^2\dd\xi<\infty$,}
\end{sloppypar}
\vspace{+0.3cm}
\item[(iii)] $u_0\in \dot{\mathcal{A}}^{-\alpha}_{2,\infty}$.
\end{itemize}
\end{Theorem}

The definition of $\dot{\mathcal{A}}^{-\alpha}_{2,\infty}$,
which is a suitable subset of the classical Besov space $\dot B^{-\alpha}_{2,\infty}(\RR^n)$,
is provided below.
Theorem~\ref{th:char} was obtained in~\cite{MR3493117} as an application
of the theory of decay characters, first introduced in~\cite{MR2493562}.
Here we would like to propose a shorter and self-contained proof of the 
above theorem, that does not require any knowledge of decay characters.

Let $\varphi$ be a smooth radial function with support contained in the annulus 
$\{\xi\in \RR^n\colon 3/4\le|\xi|\le8/3\}$, such that $\sum_{j\in\Z} \varphi(\xi/2^{-j})=1$ for all 
$\xi\in\RR^n\backslash\{0\}$.
Let $\Delta_j$ be the Littlewood--Paley localization operator around the frequency $|\xi|\simeq 2^j$, $j\in\Z$, namely, $\widehat{\Delta_j f}=\varphi(\cdot/2^{j})\widehat f$.
Let $\alpha>0$. We recall that the homogeneous Besov space 
$\dot B^{-\alpha}_{2,\infty}(\R^n)$ can be defined as the space of all tempered 
distributions~$f$, such that $f=\sum_{j\in\Z}\Delta_j f$ in $\mathcal{S}'(\R^n)$,
and, for some $C\ge0$ and all $j\in\Z$,
\begin{equation}
\label{eq:lpub}
\|\Delta_j f\|\le C2^{\alpha j}.
\end{equation}
The $\dot B^{-\alpha}_{2,\infty}$-norm is then the best constant~$C$ in  inequality~\eqref{eq:lpub}.
See~\cite[Chapt.~2]{Chemin2011}.
By definition, $\dot{\mathcal{A}}^{-\alpha}_{2,\infty}$
is the subset of $\dot B^{-\alpha}_{2,\infty}(\R^n)$ such that
the corresponding lower bound in~\eqref{eq:lpub} does hold, at least for an affine-type sequence $j_k\to-\infty$. More precisely, in addition to~\eqref{eq:lpub}, the elements
of $\dot{\mathcal{A}}^{-\alpha}_{2,\infty}$ must satisfy \\
\mbox{} \vspace{-0.750cm} \\
\begin{equation}
\label{eq:lowa}
\exists c,M>0,\;\forall k\in\N,\; \exists j_k\in [-(k+1)M,-kM]\text{ such that \;} c2^{\alpha j_k}\le  \|\Delta_{j_k} f\|_2.
\end{equation}

Let us now prove Theorem~\ref{th:char}.
\begin{proof}
``(ii)\,$\Rightarrow$(i)".
The lower bound is immediate.
Indeed, let $g(t)=(1+t)^{-1/2}$. 
For some $t_0>0$ and all $t\ge t_0$, 
we get from assumption~(ii),
$\|e^{t\Delta}u_0\|\gtrsim \int_{|\xi|\le g(t)}
|\widehat u_0(\xi)|^2\, \dd\xi\gtrsim g(t)^{2\alpha}$.

The upper bound is an application of Schonbek's Fourier splitting method~\cite{MR775190}. 
We have
$\int|\xi|^2e^{-2t|\xi|^2}|\widehat 
u_0(\xi)|^2\dd\xi\ge g(t)^2\|e^{t\Delta}u_0\|^2-g(t)^2\int_{|\xi|\le g(t)}e^{-2t|\xi|^2}|\widehat u_0(\xi)|^2\dd\xi$.
Hence, from the energy equality for the heat equation,
$
\frac{d}{dt}\|e^{t\Delta}u_0\|^2+2\int|\xi|^2e^{-2t|\xi|^2}|\widehat 
u_0(\xi)|^2\dd\xi=0
$,
we deduce
\[
\frac{d}{dt}\|e^{t\Delta}u_0\|^2
+g(t)^2\|e^{t\Delta}u_0\|^2
\le g(t)^2\int_{|\xi|\le g(t)} e^{-2t|\xi|^2}|\widehat u_0(\xi)|^2\dd\xi.
\]
By assumption (ii), the right-hand side is bounded, for some $t_0\ge0$ and
all $t\ge t_0$, by $g(t)^2g(t)^{2\alpha}=(1+t)^{-1-\alpha}$.
The upper bound for $\|e^{t\Delta}u_0\|^2$ then follows after multiplying the differential
inequality by $(1+t)^{2\alpha}$.

\medskip

``(i)\,$\Rightarrow$(iii)".
We rely on the well known fact that a tempered distribution $u_0$ belongs to the Besov space $\dot B^{-\alpha}_{2,\infty}(\R^n)$ if and only if $\sup_{t>0}t^{\alpha/2}\|e^{t\Delta}u_0\|<\infty$, the latter quantity being an equivalent norm to $\|u_0\|_{\dot B^{-\alpha}_{2,\infty}}$. See \cite[Chapter~2]{Chemin2011}.
So, assumption (i) implies that $0<\|u_0\|_{\dot B^{-\alpha}_{2,\infty}}<\infty$.
Let $t>0$ and $p\in\Z$ such that $4^{p}\le t<4^{p+1}$. We have, for some constant $c>0$
independent on~$t$,
\[
\begin{split}
0<c\le t^\alpha\|e^{t\Delta}u_0\|^2
&\le \sum_{j\in \Z}t^\alpha e^{-(2/3)^24^j t}\|\Delta_j u_0\|^2\\
&\le 4^\alpha\sum_{j\in\Z} 4^{p\alpha}e^{-(2/3)^24^j}\|\Delta_{j-p}u_0\|^2\\
&\le 4^\alpha\sum_{j\in\Z} e^{-(2/3)^24^j} 4^{\alpha j}\|u_0\|_{\dot B^{-\alpha}_{2,\infty}}^2.
\end{split}
\]
As $(e^{-(2/3)^24^j}4^{j\alpha})\in\ell^1(\Z)$, there exists $M>0$ such that
\[
4^\alpha \sum_{|j|>M} e^{-(2/3)^24^j}4^{j\alpha}<c/(2\|u_0\|_{\dot B^{-\alpha}_{2,\infty}}^2).
\]
So,
\[
4^\alpha\sum_{|j|\le M} 4^{p\alpha}e^{-(2/3)^24^j}\|\Delta_{j-p}u_0\|^2
\ge c/2.
\]
This implies the existence of a constant $\delta=\delta(\alpha,c,M)>0$ such that,
for all $p\in\Z$,
\[
\max_{|j|\le M} \|\Delta_{j-p}u_0\|^2\ge\delta 4^{-p\alpha}.
\]
This implies~\eqref{eq:lowa} and so $u_0\in\dot{\mathcal{A}}^{-\alpha}_{2,\infty}$.

\medskip
``iii)\,$\Rightarrow$ii)".
From the fact that $u_0\in \dot{\mathcal{A}}^{-\alpha}_{2,\infty}$,
we can find a constant $c>0$ and a sequence $(j_k)$ of integers, such that $j_k\to-\infty$ and for all $k$
$\|\Delta_{j_k}u_0\|_{L^2}^2\ge c 4^{j_k\alpha}$, 
with $M\equiv\|(j_k-j_{k+1})\|_{\ell^\infty}<\infty$.
Let $\rho_0=2^{j_0}$. 
For any $0<\rho\le \rho_0$, let $j_k$ be the largest integer of the sequence
$(j_k)$ such that $2^{j_k}\le \rho$.
Then, $2^{j_k}\le \rho \le 2^{j_k+M}$.
Hence,
\[
(\textstyle\frac83\rho)^{-2\alpha}\int_{|\xi|\le \frac83 \rho}|\widehat u_0(\xi)|^2\dd \xi 
\ge (\frac83)^{-2\alpha} 2^{-2 (j_k+M)\alpha} \|\Delta_{j_k} u_0\|^2
\ge c2^{-2M}(\frac83)^{-2\alpha}.
\]
This implies  the $\liminf$ condition in (ii).
Moreover, for $\rho>0$, let $j\in\Z$ such that $2^{j-1}\le\rho\le 2^{j}$. Then,
\[
\textstyle(\frac23\rho)^{-2\alpha}\int_{|\xi|\le \frac23\rho}
|\widehat u_0(\xi)|^2
\lesssim
(\textstyle\frac23\rho)^{-2\alpha}\sum_{k\le j} \|\Delta_k u_0\|_{L^2}^2
\le C\|u_0\|_{\dot B^{-\alpha}_{2,\infty}},
\]
for some constant $C>0$ independent on~$\rho$.
Hence, the  $\limsup$ condition in (ii) just follows from the fact that $u_0\in \dot B^{-\alpha}_{2,\infty}$.
\end{proof}

\medskip
\subsection{On the generecity of algebraic decays}

We will apply the previous results to the genericity problem of
algebraic estimates, from above and below, for solutions of the 
Navier--Stokes equations. 

First of all, we prove that, in the class of Leray's solutions of the unforced Navier--Stokes equations, the subclass of solutions with a $L^2$-algebraic decay
is negligible, in a topological sense.
To do this, let us consider the class $\mathcal{L}$ of all Leray's solutions
to the unforced Navier--Stokes equations. 
We endow $\mathcal{L}$ with the initial topology induced by the map 
$I\colon \mathcal{L}\to  L^2_\sigma$, 
where $I$ is the map $\boldsymbol u\mapsto \boldsymbol u_0$, 
that associates, to a solution $\boldsymbol u$, 
the corresponding initial datum.

\begin{Theorem}
\label{th:genler}
With the above topology, $\mathcal{L}$ is a Baire space 
and the
set of unforced Leray's solutions with algebraic decay 
is meager in $\mathcal{L}$.
\end{Theorem}

The second fact that we will establish is the following: 
almost all solutions (in a topological sense) 
of the unforced Navier-Stokes equations 
with $L^2$-algebraic
decay $O((1+t)^{-\alpha})$, with $0<\alpha<(n+2)/4$, 
do satisfy the two sided bounds
$(1+t)^{-\alpha}\lesssim 
\|\hspace{+0.020cm} \mbox{\boldmath $u$}(\cdot,t) \hspace{+0.020cm}
\|_{L^2(\RR^n)}\lesssim (1+t)^{-\alpha}$.
To make this rigorous, for $0<\alpha<(n+2)/4$, 
we introduce
the class $\mathcal{D}_{\hspace{-0.010cm}\alpha}$
of all Leray's solution of the unforced 
Navier--Stokes equations 
such that
\[
\|\hspace{+0.020cm}\boldsymbol u(\cdot,t)
\hspace{+0.020cm}\|_{L^2}\lesssim \hspace{+0.020cm}
(1+t)^{-\alpha}.
\]
Let us endow $\mathcal{D}_{\hspace{-0.010cm}\alpha}$ 
with a natural topology.
By Theorem~\ref{Inverse_Wiegner}, 
we know that 
$\boldsymbol u\in \mathcal{D}_{\hspace{-0.010cm}\alpha}$ 
implies 
$\|\hspace{+0.030cm}e^{t\Delta}\boldsymbol u_0
\hspace{+0.030cm}\|_{L^2}=O((1+t)^{-\alpha})$.
The Besov space theory~\cite[Chapter 2]{Chemin2011} 
then implies that 
$\boldsymbol u_0\in \dot B^{-2\alpha}_{2,\infty}(\RR^n)$.
So, we can endow $\mathcal{D}_{\hspace{-0.010cm}\alpha}$ 
with the initial topology induced by the map 
$I\hspace{+0.010cm}\colon \hspace{+0.040cm}
\mathcal{D}_\alpha\to\dot B^{-2\alpha}_{2,\infty}\cap L^2_\sigma(\RR^n)$, 
where $I(\boldsymbol u)=\boldsymbol u_0$, as before, 
and
the topology of $\dot B^{-2\alpha}_{2,\infty}\cap L^2_\sigma(\RR^n)$ 
is that induced by the norm 
$\boldsymbol u_0\mapsto\|\boldsymbol u_0\|_{\dot B^{-2\alpha}_{2,\infty}}+\|\boldsymbol u_0\|_{L^2_\sigma}$.

Then we have the following:
\begin{Theorem}
\label{th:genler2}
Let $0<\alpha<(n+2)/4$. 
The set of Leray's solutions of the unforced Navier-Stokes equations, 
such that the two sided estimates
$(1+t)^{-\alpha}\lesssim \hspace{+0.010cm}\|\boldsymbol u(\cdot,t)
\|_{L^2}\lesssim (1+t)^{-\alpha} \hspace{-0.020cm}$
do hold, 
is residual in the Baire space~$\mathcal{D}_{\hspace{-0.010cm}\alpha}$.
\end{Theorem}

The above theorems rely on the two propositions below
and on a simple lemma in general topology which we now state.

\begin{Lemma}
\label{lem:baire}
Let $I\colon X\to Y$ be a surjective map from a set $X$ to a topological space~$Y$.
Endow $X$ with the initial topology of the map $I$. 
Then:
\begin{itemize}
\item[-] If $Y$ is a Baire space, then $X$ is a Baire space.
\item[-] If $B\subset Y$ is residual (resp. meager) in $Y$, then $I^{-1}(B)$ is residual (resp. meager) in $X$.
\end{itemize}

\end{Lemma}

\begin{proof}
The open sets of $X$ are, by definition, the sets of the form
$I^{-1}(U)$, where $U$ is open in $Y$.

Now, let $V\subset Y$ and $A=I^{-1}(V)$.
Let us first observe the following simple fact:
\[
\text{
$A$ open and dense in $X$ $\iff$  
$V$ open and dense in $Y$}.
\tag{OD}
\]

Indeed,
if $V$ is an open dense set in $Y$ and $A$ were not dense, 
then there would exist a non-empty open set $W$ of $X$ 
such that $A\cap W$ is empty. But $W=I^{-1}(U)$ 
for some open set $U$ of $Y$. 
Therefore, $I^{-1}(V\cap U)=A\cap W=\emptyset$.
But $I$ is surjective, hence $V\cap U=\emptyset$ 
and so 
$U = \emptyset$ because $V$ is dense. 
Thus,
$W=\emptyset$ and we get a contradiction.

Conversely, if $A$ is open and dense in $X$, 
then there exists an open set $U$ in $Y$ such that 
$I^{-1}(U)=A=I^{-1}(V)$. 
By surjectivity, $U=V$ and so $V$ is open.  
But $V$ must be dense in $Y$, otherwise there would exist
$W\not=\emptyset$ open in $Y$ such that $V\cap W$ is empty. 
Hence, by the surjectivity, $\emptyset=I^{-1}(V\cap W)=A\cap I^{-1}(W)$. 
This is a contradiction because
$I^{-1}(W)$ is open, non-empty, and $A$ is dense in~$X$.

For proving the first assertion of the lemma, 
let $(A_n)$ be a countable collection
of open dense sets in $X$. 
Let us prove that $\bigcap_n A_n$ is dense in~$X$.
Indeed, let 
$W\not=\emptyset$ be an open set in~$X$.
We have $A_n=I^{-1}(V_n)$
where  $V_n$ is open and dense in $Y$ by (OD). 
Moreover
$W=I^{-1}(V)$ with $V$ open and non-empty in $Y$. 
So, $W\cap(\bigcap_n A_n)=
I^{-1}(V\cap(\bigcap_n V_n))$. But $V\cap(\bigcap_n V_n)\not=\emptyset$
by the assumption that $Y$ is a Baire space (and so $\bigcap_n V_n$ is dense
in $Y$). Moreover, $I$ is surjective, hence $W\cap(\bigcap_n A_n)\not=\emptyset$.
This proves that $\bigcap_n A_n$ is dense in $X$ and so $X$ is a Baire space.
    
Let us prove the second assertion of the lemma. If $B$ is residual in $Y$, then
$B$ contains a countable intersection $\bigcap U_n$ of open and dense sets $U_n$
in $Y$.
Then $I^{-1}(B)$ contains $\bigcap_n I^{-1}(U_n)$, which is a countable intersection
of open dense sets of $X$ by (OD). So $I^{-1}(B)$ is residual in~$X$. The conclusion for meager sets is obtained passing to the complementary.
\end{proof}

\begin{Proposition}
\label{prop:meag}
The subspace of $L^2_{\sigma}(\R^n)$-initial data, such that there
exists a Leray solution of the unforced Navier--Stokes equations with
an algebraic decay of the energy, is a meager set in $L^2_{\sigma}(\RR^n)$.
\end{Proposition}

\begin{proof}
Indeed, if $\boldsymbol u_0\in L^2_\sigma(\R^n)$ is such that there
exists a Leray solution of the unforced Navier--Stokes equations with
an algebraic decay of the energy, then, by the inverse Wiegner theorem,
we must have $\|e^{t\Delta}\boldsymbol u_0\|^2\lesssim (1+t)^{-\alpha}$,
for some $\alpha>0$.
But,
\[
\limsup_{\rho\to0+} \rho^{-2\alpha}
 \int_{|\xi|\le \rho}|\widehat{\boldsymbol u}_0(\xi)|^2\dd\xi<\infty
\iff
\|e^{t\Delta}\boldsymbol u_0\|^2 \lesssim (1+t)^{-\alpha}.
\]
The implication ``$\Rightarrow$'' was established in the first part of the proof of  Theorem~\ref{th:char} and the
implication ``$\Leftarrow$'' is obvious, since
$\int_{|\xi|\le \rho} |\widehat{\boldsymbol u}_0(\xi)|^2 
\lesssim  \|e^{\Delta/\rho^2}\boldsymbol u_0\|^2\lesssim \rho^{2\alpha}$.

For any $\alpha,\rho,K>0$,
let us consider the closed sets of $L^2_\sigma$
\[
F_{\alpha,\rho,K}
=\bigl\{\boldsymbol u_0\in L^2_{\sigma}\colon 
\int_{|\xi|\le \rho}|\widehat{\boldsymbol u}_0|^2\le K\rho^{2\alpha}\bigr\}
\quad\text{and}\quad
A_{\alpha,\rho_0,K}=\bigcap_{0<\rho\le \rho_0} F_{\sigma,\rho,K}.
\]
Let us show that $A_{\alpha,\rho_0,K}$ has an empty interior in $L^2_\sigma$.
Indeed, let $\boldsymbol v_0\in L^2_\sigma$, given by
\[
\widehat{\boldsymbol v}_0(\xi)=(-i\xi_2,i\xi_1,0,\ldots,0) \,|\xi|^{-1-n/2}(\log|\xi|)^{-1},
\]
if $|\xi|\le 1/2$ and $\widehat{\boldsymbol v}_0(\xi)=0$ otherwise.
Here $i$ is the imaginary unit.
The divergence-free condition is ensured by the fact that $\xi\cdot\widehat{\boldsymbol v}_0(\xi)=0$.
Then, for all $0<\rho\le 1/2$, 
we have
$\int_{|\xi|\le \rho} |\widehat{\boldsymbol v}_0|^2\dd \xi\approx 
\int_0^\rho r^{-1}(\log r)^{-2}\dd r\approx |\log\rho|^{-1}$.
Now, let $\boldsymbol u_0\in A_{\alpha,\rho_0,K}$ and $\epsilon>0$.
Then $\boldsymbol u_0+\epsilon\boldsymbol v_0\not\in A_{\alpha,\rho_0,K}$.
Indeed, for some $\rho_1$, with $0<\rho_1\le \rho_0$ and all $0<\rho\le \rho_1$:
\[
\begin{split}
\bigl(\int_{|\xi|\le \rho} |\widehat{\boldsymbol u_0\!+\!\epsilon \boldsymbol v_0}|^2(\xi)\Bigr)^{1/2}
&\ge
\Bigl|\epsilon \Bigl(\int_{|\xi|\le \rho}|\widehat{\boldsymbol v}_0(\xi)|^2\dd \xi\Bigr)^{1/2}-\Bigl(\int_{|\xi|\le\rho}|\widehat{\boldsymbol u}_0(\xi)|^2\dd \xi\Bigr)^{1/2}\Bigr|\\
&\gtrsim \Bigl|\epsilon|\log\rho|^{-1/2}-\sqrt K\rho^\alpha\Bigr|\\
&\gtrsim \epsilon|\log\rho|^{-1/2}.
\end{split}
\]
Therefore,
\[
\rho^{-2\alpha}\int_{|\xi|\le \rho}|\widehat{\boldsymbol u_0\!+\!\epsilon\boldsymbol v_0}|^2(\xi)\dd \xi
\to+\infty\qquad\text{as $\rho\to0+$},
\]
which proves that $\boldsymbol u_0+\epsilon\boldsymbol v_0\not\in A_{\alpha,\rho_0,K}$.

\medskip
The set of initial data giving rise to Leray's solutions of the Navier--Stokes equations with algebraic decay, which can be characterized as
\begin{equation}
\label{eq:setA}
A=\bigl\{ \boldsymbol u_0\in L^2_\sigma\colon \exists\alpha>0,
\;\textstyle\limsup_{\rho\to0+}\rho^{-2\alpha}\int_{|\xi|\le \rho}|\widehat{\boldsymbol u}_0(\xi)|^2\dd\xi<+\infty\bigr\},
\end{equation}
can be written as 
\[
A=\bigcup_{\alpha,\rho_0,K>0} A_{\alpha,\rho_0,K}.
\]
By the previous discussion, $A_{\alpha,\rho_0,K}$ is a closed set with empty interior in $L^2_\sigma$.
But a countable union of these sets is enough to cover $A$, so $A$ is meager
in~$L^2_\sigma$.
\end{proof}

Our next proposition shows that, for $\alpha>0$,  initial data in
$\dot B^{-2\alpha}_{2,\infty}\cap L^2_\sigma(\RR^n)$
generically do belong to $\dot{\mathcal{A}}^{-2\alpha}_{2,\infty}\cap L^2_\sigma(\R^n)$.
To this purpose, we endow $\dot B^{-2\alpha}_{2,\infty}\cap L^2_\sigma(\RR^n)$ with the natural Banach norm
\[
w\mapsto\|w\|_{\dot B^{-2\alpha}_{2,\infty}(\RR^n)}+\|w\|_{L^2(\RR^n)}.
\]

\begin{Proposition}
\label{prop:residual}
The set
$\dot{\mathcal{A}}^{-2\alpha}_{2,\infty}\cap L^2_\sigma(\R^n)$
is residual in $\dot B^{-2\alpha}_{2,\infty}\cap L^2_\sigma(\R^n)$.
\end{Proposition}

\begin{proof}
For any $j_0\in\Z$, a norm in $\dot B^{-2\alpha}_{2,\infty}\cap L^2_\sigma(\RR^n)$ equivalent to the previous natural norm is
\begin{equation}
\label{eqnorm}
\boldsymbol w\mapsto\|\boldsymbol w\|_{\alpha,j_0}:= \sup_{j\le j_0} \Bigl( 2^{-2\alpha j}\|\Delta_j \boldsymbol w\| \Bigr) 
+ \Bigl(\sum_{j>j_0}^\infty \|\Delta_j \boldsymbol w\|^2\Bigr)^{1/2}.
\end{equation}

Let
\begin{multline}
\label{eq:V}
V_\alpha=\bigl\{\boldsymbol u_0\in \dot B^{-2\alpha}_{2,\infty}\cap L^2_\sigma(\RR^n)\colon \\
\exists \delta>0, \exists j_0\in \ZZ \text{ such that, for all $j\le j_0$, }
\|\Delta_j \boldsymbol u_0\| \ge \delta 2^{2\alpha j} \bigr\}.
\end{multline}
Notice that $V_\alpha\subset\dot{\mathcal{A}}^{-2\alpha}_{2,\infty}\cap L^2_\sigma(\R^n)$, so it is enough to prove that $V_\alpha$ is a dense open
set in $B^{-2\alpha}_{2,\infty}\cap L^2_\sigma(\RR^n)$.

Let $\boldsymbol u_0\in V_\alpha$, $\delta>0$ and $j_0\in \ZZ$ such that
 $\|\Delta_j \boldsymbol u_0\| \ge \delta 2^{2\alpha j}$ for all $j\le j_0$.
Then the  $\|\cdot\|_{\alpha,j_0}$-ball of $\dot B^{-2\alpha}_{2,\infty}\cap L^2_\sigma(\RR^n)$ centered at $\boldsymbol u_0$ with radius $\delta/2$ is contained in
$V_\alpha$. So $V_\alpha$ is open in $B^{-2\alpha}_{2,\infty}\cap L^2_\sigma(\RR^n)$.

If $\boldsymbol u_0\in \dot B^{-2\alpha}_{2,\infty}\cap L^2_\sigma(\RR^n)$ and $\epsilon>0$, then we define,
for an arbitrary choice of $j_0\in \ZZ$ and any $j\in\ZZ$, 
\[
\boldsymbol w_j=
\left\{
\begin{aligned}
&\Delta_j \boldsymbol u_0, \qquad\qquad\qquad\qquad\,\text{for all $j>j_0$ or ($j\le j_0$ and $2^{-2\alpha j}\|\Delta_j u_0\|\ge \epsilon$)}\\
&\epsilon\, 2^{2\alpha j-nj/2}
\mathcal{F}^{-1}\bigl(\textstyle\frac{(-i\xi_2,i\xi_1,0,\ldots,0)}{|\xi|}{\bf 1}_{\{ 2^j\le |\xi|\le 2^{j+1}\}}\bigr)  \qquad\qquad\qquad\qquad\text{otherwise},
\end{aligned}
\right.
\]
where
$\mathcal{F}$ is the Fourier transform and ${\bf 1}_X$ the indicator function of the set~$X$.
We have $\sup_{j\in\Z}2^{-2\alpha j}\|\boldsymbol w_j\|<\infty$ and $\mathcal{F}\boldsymbol w_j$
is supported in a dyadic annulus of radii $\approx 2^j$,
Hence, $\boldsymbol w:=\sum_{j\in\ZZ} \boldsymbol w_j$ belongs to $\dot B^{-2\alpha}_{2,\infty}(\R^n)$. (This follows from 
\cite[Lemma~2.23 and Remark~2.24]{Chemin2011}).
Moreover, all these functions $\boldsymbol w_j$ are divergence free.
So one also easily checks that $\boldsymbol w\in L^2_\sigma(\R^n)$.

Moreover, for some $c>0$ and for all $j\le j_0$, we have 
$2^{-2\alpha j} \|\Delta_j \boldsymbol w\|\ge c\,\epsilon$. So, in fact, $w\in V_\alpha$.
But, for all integer $j\le j_0$, we have $2^{-2\alpha j}\|\Delta_j(\boldsymbol u_0-\boldsymbol w)\|\le 2\epsilon$, which proves that the $\|\cdot\|_{\alpha,j_0}$-distance
between $u_0$ and $w$,
is less than $2\epsilon$.
Hence, $V_\alpha$ is dense in $B^{-2\alpha}_{2,\infty}\cap L^2_\sigma(\R^n)$.
\end{proof}

Proposition~\ref{prop:residual} could be reformulated in terms of decay characters.
Indeed, the elements of~${\mathcal{A}}^{-2\alpha}_{2,\infty}\cap L^2_\sigma(\R^n)$,
are precisely those for which a decay character (equal to~$\alpha-n/2$) does exist. 
See~\cite[Proposition~4.3]{MR3493117}.

We are now in the position proving Theorem~\ref{th:genler} and Theorem~\ref{th:genler2}.

\begin{proof}[Proof of Theorem~\ref{th:genler}]
Let us consider the class $\mathcal{L}$ of all Leray's solutions
to the unforced Navier--Stokes equations, endowed 
with the initial topology induced by the map 
$I\colon \mathcal{L}\to  L^2_\sigma$,
where $I$ is the map $\boldsymbol u\mapsto \boldsymbol u_0$ 
that associates, to a solution $\boldsymbol u$, the corresponding initial datum. 
This map is surjective, owing to Leray's theorem. 
(It is not known yet if the map $I$ is injective or not).
But $L^2_\sigma$ is a complete metric space, hence combining Baire's theorem 
with Lemma~\ref{lem:baire} implies that $\mathcal{L}$
is a Baire space. 
Moreover, the set of unforced Leray's solutions with algebraic decay 
is the set $I^{-1}(A)$, where 
$A\subset L^2_\sigma$ is the set introduced in Eq.~\eqref{eq:setA}.
But~$A$ is meager by Proposition~\ref{prop:meag}.
  This set $I^{-1}(A)$ is then meager in
$\mathcal{L}$, because of Lemma~\ref{lem:baire}.
\end{proof}

\begin{proof}[Proof of Theorem~\ref{th:genler2}]
Let us first observe that the solution-to-datum map $I\colon 
\mathcal{D}_\alpha\to\dot B^{-2\alpha}_{2,\infty}\cap L^2_\sigma(\RR^n)$ 
is surjective.
Indeed, if $\boldsymbol u_0\in  \dot B^{-2\alpha}_{2,\infty}(\RR^n)$, 
then $\|e^{t\Delta}\boldsymbol u_0\|=O(t^{-\alpha})$ 
by the Besov space theory~\cite[Chapter 2]{Chemin2011}. 
But then, 
for $\boldsymbol u_0 \in L^2_\sigma
\cap  \dot B^{-2\alpha}_{2,\infty}(\RR^n)$, 
Wiegner's theorem
ensures that a solution $\boldsymbol u\in \mathcal{D}_\alpha$, 
arising from $\boldsymbol u_0$
does exist.
The class $\mathcal{D}_\alpha$ is a Baire space, 
owing to Lemma~\ref{lem:baire} 
and and to the completeness of 
$L^2_\sigma \cap  \dot B^{-2\alpha}_{2,\infty}(\RR^n)$.

Combining Proposition~\ref{prop:residual} 
with Lemma~\ref{lem:baire}, we see that
$I^{-1}(\dot{\mathcal{A}}^{-2\alpha}_{2\infty}\cap L^2_\sigma)$ 
is residual in $\mathcal{D}_{\hspace{-0.010cm}\alpha}$.
But, by Theorem~\ref{th:char} and Corollary~\ref{cor:tsb},
$I^{-1}(\dot{\mathcal{A}}^{-2\alpha}_{2\infty}\cap L^2_\sigma)$ 
is precisely the set of Leray's solutions 
of the unforced Navier-Stokes equations 
such that the two sided estimates
$(1+t)^{-\alpha}\lesssim
\|\boldsymbol u(\cdot,t)\|_{L^2}\lesssim (1+t)^{-\alpha}\hspace{-0.042cm}$
do hold.
\end{proof}

Due to Remark~\ref{remark1}, for $2\le n\le 4$, the topological results of the present section could can be naturally extended to higher derivatives.


%
%
%
%

\bibliographystyle{plain}
\bibliography{BrandolesePerusatoZinganoBiblio}{}

\mbox{} \vspace{-0.550cm} \\
%
%
	
\end{document}